\documentclass[12pt]{amsart}
\usepackage{fullpage}
\usepackage[english]{babel}
\usepackage{mathtools}
\usepackage{amssymb}
\usepackage{amsthm}
\usepackage{mathrsfs}
\usepackage{nicefrac}
\usepackage{upgreek}
\usepackage[T1]{fontenc}
\usepackage{graphicx}
\usepackage{hyperref}
\usepackage{stmaryrd}
\usepackage{graphicx}
\usepackage{booktabs}
\usepackage{hyperref}

\usepackage[all]{xy}
\usepackage{graphics}
\usepackage{color}
\usepackage[T1]{fontenc}

\usepackage[backend=bibtex,firstinits=true,style=numeric,maxbibnames=99,maxalphanames=99,isbn=false,url=false]{biblatex}\renewbibmacro{in:}{}
\usepackage{enumitem}

\makeatletter
\newtheorem*{rep@theorem}{\rep@title}
\newcommand{\newreptheorem}[2]{%
\newenvironment{rep#1}[1]{%
 \def\rep@title{#2 \ref{##1}}%
 \begin{rep@theorem}}%
 {\end{rep@theorem}}}
\makeatother

\newcommand{\C}[1]{{\mathcal #1}}

\newtheorem{theorem}{Theorem}[section]
\newtheorem{lemma}[theorem]{Lemma}

\newtheorem{conjecture}[theorem]{Conjecture}
\theoremstyle{definition}\newtheorem{definition}[theorem]{Definition}
\theoremstyle{remark}\newtheorem{remark}[theorem]{Remark}
\theoremstyle{definition}
\makeatletter\@addtoreset{case}{example}\makeatother
\theoremstyle{definition}\newtheorem{claim}{Claim}

\begin{document}

\title{Elementary bounded generation for ${\rm SL}_n$ for global function fields and $n\geq 3$}

\author{Alexander A. Trost}
\address{Fakult\"{a}t f\"{u}r Mathematik, Ruhr Universit\"{a}t Bochum, D-44780 Bochum, Germany}
\email{Alexander.Trost@ruhr-uni-bochum.de}

\begin{abstract}
This paper shows that the group ${\rm SL}_n(R)$ is boundedly elementary generated for $n\geq 3$ and $R$ the ring of algebraic integers in a global function field. Contrary to previous statements for number fields and earlier statements for global function fields, the bounds proven in this preprint for elementary bounded generation are independent of the underlying global function field and only depend on the integer $n.$ Combining our main result with earlier results, we further establish that elementary bounded generation always has bounds independent from the global field in question, only depending on $n.$
\end{abstract}

\maketitle

\section{Introduction} 
\label{intro}

Classically, a group $G$ is called \textit{boundedly generated} if there are finitely many cyclic groups $Z_1,\dots,Z_L$ such that $G=Z_1\cdots Z_L.$ This property was introduced, among else, as a possible unified approach to proving the Congruence Subgroup Property of Arithmetic Groups as stated in the (now disproven) Rapinchuk-Conjecture \cite[Question~A]{corvaja2021nonvirtually} and has applications concerning the representation theory of the group in question. There are papers proving this property for various classical, arithmetic groups like ${\rm SL}_n(\mathbb{Z})$ in a paper by Carter and Keller \cite{MR704220} or for more general arithmetic (twisted) Chevalley groups by Tavgen \cite{MR1044049}. For ${\rm SL}_2(R)$ for $R$ a ring of algebraic integers with infinitely many units there are papers by Carter, Keller and Paige written up by Morris \cite{MR2357719} and another one by Rapinchuk, Morgan and Sury \cite{MR3892969}. More recently, there was also a result for certain isotropic orthogonal groups over number fields \cite{MR2228948}. In case of Chevalley groups like ${\rm SL}_n$ or ${\rm Sp}_{2n}$, bounded generation results are usually proven by showing first that each element of the group can be written as some bounded product of elementary matrices (or the corresponding generalization in the group, so called \textit{root elements}) and then analyzing the subgroups given by root elements more closely. The first property is commonly referred to as \textit{bounded generation by root elements} or \textit{bounded elementary generation}. For linear groups in positive characteristic classical bounded generation is only possible, if the group is virtually abelian \cite[Theorem~1]{MR2022116}. That being said, the weaker property of bounded generation by root elements still holds in positive characteristic for Chevalley groups defined over global function fields:

\begin{theorem}\cite[Theorem~1.3]{Chevalley_positive_char_tentative}
Let $K$ be a global function field that is a finite extension of $\mathbb{F}_q(T)$ with ring of integers $\C O_K$ and field of constants $\mathbb{F}_q.$ Further, let $\Phi$ be an irreducible root system of rank at least $2.$ Then the simply-connected, split Chevalley group $G(\Phi,\C O_K)$ is boundedly generated by root elements. More precisely, there is a constant $L(q,k)\in\mathbb{N}$ only depending on $q$ and $k:=[K:\mathbb{F}_q(T)]$ such that each element of $G(\Phi,\C O_K)$ can be written as a product of $L(q,k)\cdot{\rm rank}(\Phi)$ many root elements.  
\end{theorem}

There are also some earlier results: The earliest known to me is the result \cite[Theorem~2]{MR379441} by Clifford and Queen. More recently, Nica has presented the following result:

\begin{theorem}\cite[Theorem~2]{Nica}\label{Nica_thm}
Let $\mathbb{F}$ be a finite field and $n\geq 3.$ Then ${\rm SL}_n(\mathbb{F}[T])$ is boundedly generated by root elements.
\end{theorem}

However, most such results concerning bounded generation by root elements depend on the global function field in question: Cliffords requires rather artificial conditions on the class number of the global function field and the degrees of certain primes and Nica's result crucially uses the fact that $\mathbb{F}[T]$ is a principal ideal domain. The main goal of this preprint is to generalize such results to all rings of integers in global function fields in a way that avoids the dependencies on the function field in question: 

\begin{theorem}\label{main_thm}
Let $K$ be a global function field with ring of integers $\C O_K$ and $n\geq 3$. Then ${\rm SL}_n(\C O_K)$ is boundedly generated by elementary matrices. More precisely, each element of ${\rm SL}_n(\C O_K)$ can be written as a product of at most $1/2(3n^2-n)+50$ elementary matrices.
\end{theorem}

\begin{remark}
For the sake of completeness, we should mention that Carter and Keller claim in \cite[Remarks]{MR704220} that their (at the time unpublished) manuscript \cite{carter1985congruence} also contains a result for rings of integers in global function fields. The publication \cite{carter1985congruence} does not seem to be available though and the problem of bounded elementary generation in this case is still considered to be open in much later publications like \cite[P.~404]{Nica}, \cite[P.~523]{stepanov2016structure} and \cite[9.~Final Remarks]{https://doi.org/10.48550/arxiv.2204.10951}.
\end{remark}

Similar to Nica's Theorem~\ref{Nica_thm}, Theorem~\ref{main_thm} arises from an adaption of the method of Carter-Keller in \cite{MR704220} to the global function field setup. The resulting bounds in Theorem~\ref{main_thm} do not depend on the global function field in question however. The underlying reason for this difference between global function fields and number fields seems to be that for a number field $K$ the Carter-Keller argument required handling the ramified primes of the field extension $K|\mathbb{Q}$ one at a time in order to ensure that various different cyclotomic field extensions of $K$ appearing in the proof are in a certain sense compatible with one another by controlling their ramification behavior. In contrast, for global function fields, one is not concerned with cyclotomic but rather constant field extensions of the function field in question. These constant field extensions (if chosen correctly) are almost automatically compatible in the required sense as they are always unramified. The structure of the paper is straightforward: In Section 2, we introduce the necessary background on global function fields, their constant field extensions and class field theory and in Section 3, we prove Theorem~\ref{main_thm}. In the last section, we sketch how these arguments can be used to answer questions from \cite{MR3892969} and \cite{https://doi.org/10.48550/arxiv.2204.10951} concerning elementary bounded generation and pose a conjecture.

\section*{Acknowledgments}

The first part of the research for this preprint was done during a research visit at the Mathematische Forschungsinstitut Oberwolfach and the second part during a visit at the Otto-von-Guerricke University Magdeburg. I am very grateful for the support I received at these places.

\section{Basic definitions and notions}
\label{sec_basic_notions}

\subsection{Global function fields and their integers}

In this subsection, we introduce global function fields, their rings of integers and their primes. For a good introduction to the topic, we refer the reader to Rosen's book \cite{MR1876657}. We begin with the definition of global function fields:

\begin{definition}
A \textit{global function field over $\mathbb{F}_q$} is a finite field extension $K$ of the rational function field $\mathbb{F}_q(T)$ for $\mathbb{F}_q$ a finite field with $q$ elements. Then the ring of integers $\C O_K$ of $K$ is defined as the integral closure of $\mathbb{F}_q[T]$ in $K.$
\end{definition}

\begin{remark}
\hfill
\begin{enumerate}
\item{Equivalently, one can define global function fields as the field of functions $\mathbb{F}_q(C)$ of a nonsingular, geometrically integral, affine curve $C$ defined over the finite field $\mathbb{F}_q.$}
\item{The principal example of a global function over $\mathbb{F}_q$ is the field of rational functions $\mathbb{F}_q(T).$ The associated curve is the affine line $\mathbb{A}^1$ of the field $\mathbb{F}_q.$}
\end{enumerate}
\end{remark}

In the rest of the paper, we will always replace the field $\mathbb{F}_q$ by its algebraic closure $\mathbb{F}$ in $K$ and consequently set $q:=|\mathbb{F}|$. We leave it as an exercise to check that $\mathbb{F}$ is a finite field extension of $\mathbb{F}_q$. Assuming that $\mathbb{F}$ is algebraically closed in $K$, we call $\mathbb{F}$ the \textit{field of constants of $K$}. Later, we will also need the concept of primes in global function fields, not to be confused with the prime ideals of the corresponding ring of integers:
 
\begin{definition}
Let $K$ be a global function field with field of constants $\mathbb{F}$. 
\begin{enumerate}
\item{Then a \textit{prime of $K$} is a discrete valuation domain $R_{\C P}$ with fraction field $K$ such that $R_{\C P}$ contains $\mathbb{F}$ and has maximal ideal $\C P.$}
\item{We define the function
\begin{equation*}
{\rm ord}_{\C P}:K^*\to\mathbb{Z}
\end{equation*}  
as follows: If $a\in R_{\C P}-\{0\}$ is given, then ${\rm ord}_{\C P}(a)\in\mathbb{Z}$ is defined as the maximal $n\in\mathbb{N}_0$ such that $a\in\C P^n.$ If $a\in K^*$ is not an element of $R,$ then ${\rm ord}_{\C P}(a)$ is defined as $-{\rm ord}_{\C P}(a^{-1}).$}
\end{enumerate}
\end{definition}

\begin{remark}
Especially later on, we will often refer to `the prime $\C P$ of $K$' rather than `to the prime $R_{\C P}$ of $K$.'
\end{remark}

Furthermore, for a finite extension $K\mid L$ of global function fields and $R_{\C P}$ a prime of $K,$ the intersection $S_p:=R_{\C P}\cap L$ is a prime of the global function field $L.$ In this case, we refer to the prime $\C P$ of $K$ as covering (or lying over) the prime $p$ of $L.$ Let $x_p\in S_p$ be the generator of the maximal ideal $p$ in $S_p.$ We then define $e(\C P,p):={\rm ord}_{\C P}(x_p)$. The integer $e(\C P,p)$ does not depend on the choice of $x_p$ and is called the \textit{ramification index of $\C P$ over $p.$} Furthermore, the integer $f(\C P\mid p):=[R_{\C P}/\C P:S_p/p]$ is called \textit{the relative degree of $\C P$ over $p.$} 

As mentioned already, the field $\mathbb{F}(T)$ is a global function field in its own right. It has two distinct types of primes: First, for a monic, irreducible polynomial $f(T)\in\mathbb{F}[T]$, consider the localization $R_f:=\mathbb{F}[T]_{f(T)\mathbb{F}[T]}.$ This local ring is a prime of $\mathbb{F}(T).$ Second, consider the ring $A':=\mathbb{F}[T^{-1}]$ and its localization $R_{\infty}:=A'_{T^{-1}A'}.$ The ring $R_{\infty}$ is also a prime of $\mathbb{F}(T)$ and is commonly called \textit{the prime $\C P_{\infty}$ at infinity of $\mathbb{F}(T)$}. We leave it as an exercise to the reader to check that these two types make up all primes of $\mathbb{F}(T)$.

Also note, that each maximal ideal $\mathfrak{P}$ of the ring of integers $\C O_K$ of a global function field $K$ defines a prime of $K$, namely the localization $R_{\mathfrak{P}}:=(\C O_K)_{\mathfrak{P}}.$ We leave it as an exercise to check that the primes of $K$ arising in this manner are precisely the ones not covering the prime $\C P_{\infty}$ of $\mathbb{F}(T)$ and that the map 
\begin{equation*}
{\rm MaxSpec}(\C O_K)\to\{\C P\text{ prime of }K\mid \C P\text{ does not cover }\C P_{\infty}\},\mathfrak{P}\mapsto R_{\mathfrak{P}}
\end{equation*}
is a bijection. By way of this bijection, we will usually identify maximal ideals of $\C O_K$ with their associated primes.
\begin{remark}\label{choiceS}
Identifying maximal ideals in $\C O_K$ and certain primes of $K$ is less natural than it might seem: Contrary to number fields, where there is a canonical choice for the infinite primes and so the identification between maximal ideals of $\C O_K$ and the finite primes is canonical, for global function one has a choice in concerning which set $S$ of primes are considered to be at infinity and a different choice of $S$ yields a ring of integers $\C O_K^{(S)}$ different from $\C O_K$. Our implicit choice of $S$ amounts to those primes in $K$ covering the prime $\C P_{\infty}$ in $\mathbb{F}(T).$ This choice recovers $\mathbb{F}[T]$ as the ring of integers in $\mathbb{F}(T)$ and $\C O_K$ as the normal closure of $\mathbb{F}[T]$ in $K$ analogous to the situation for number fields where $\mathbb{Z}$ is the ring of integers in $\mathbb{Q}$ and $\C O_K$ its normal closure. It would be possible to state the results of the preprint for $\C O_K^{(S)}$ for other sets $S$ but as long as $S\neq\emptyset$ holds this would not change the statement of Theorem~\ref{main_thm} nor the argument while introducing more notation, so we decided against doing so.
\end{remark}

\subsection{Class field theory}

In this subsection, we will give a brief introduction to results and definitions from class field theory of global function fields needed later on. The theory is very similar to the corresponding theory of algebraic number fields, but there are technical differences related to the absence of infinite places (e.g. Remark~\ref{choiceS}). Let $K$ be a global function field and $\mathfrak{m}$ a non-zero ideal in $\C O_K$. Then define 
\begin{equation*}
J^{\mathfrak{m}}_K:=\{\mathfrak{a}\mid\mathfrak{a}\text{ non-zero fractional ideal in }\C O_K\text{ coprime to }\mathfrak{m}\}.
\end{equation*} 
Recall that a \textit{fractional ideal in $\C O_K$} is an $\C O_K$-submodule $\mathfrak{a}$ of $K$ such that there is a $\lambda\in\C O_K$ with $\lambda\mathfrak{a}\subset\C O_K.$ Also a fractional ideal $\mathfrak{a}$ is called \textit{coprime to $\mathfrak{m}$}, if its prime factorization shares no prime divisor with $\mathfrak{m}$ including negative powers $\C Q^{-k}$ appearing as prime divisors of $\mathfrak{m}$. We consider $J^{\mathfrak{m}}_K$ as a group with multiplication given by products of fractional ideals. Further, set  
\begin{equation*}
P^{\mathfrak{m}}_K:=\{a\C O_K\mid a\in K^*,a\C O_K\in J^{\mathfrak{m}}_K\}.
\end{equation*} 
Then the quotient ${\rm Cl}^{\mathfrak{m}}_K:=J^{\mathfrak{m}}_K/P^{\mathfrak{m}}_K$ is called the \textit{ray class group modulo $\mathfrak{m}.$} 
Next, consider an Galois extension of global function fields $L|K$ and let $\C P$ an unramified prime in $K.$ Then let $\mathfrak{P}$ be a prime in $L$ lying over $\C P$ and let $S_{\mathfrak{P}}\subset L$ be the discrete valuation ring associated with the prime $\mathfrak{P}$. Then there is an element $(\mathfrak{P},L|K)\in{\rm Gal}(L|K)$ defined uniquely by the property 
\begin{equation*}
\forall a\in S_{\mathfrak{P}}:(\mathfrak{P},L|K)(a)\equiv a^{N(\C P)-1}\text{ mod }\mathfrak{P}.
\end{equation*} 
For two primes $\mathfrak{P}_1,\mathfrak{P_2}$ lying over $\C P$, the two elements $(\mathfrak{P}_1,L|K),(\mathfrak{P}_2,L|K)$ of ${\rm Gal}(L|K)$ are conjugate to each other in ${\rm Gal}(L|K).$ Hence for an abelian extension $L|K$, $(\mathfrak{P},L|K)$ does not depend on the choice of $\mathfrak{P}$ and so we set $(\C P,L|K):=(\mathfrak{P},L|K).$ For a fractional ideal $\mathfrak{a}\subset K$ factorizing as $\mathfrak{a}=\prod_{\C P} \C P^{n_{\C P}}$ with primes $\C P$ unramified in an abelian extension $L|K$, one defines  
\begin{equation*}
(\mathfrak{a},L|K):=\prod_{\C P}(\C P,L|K)^{n_{\C P}}\in{\rm Gal}(L|K)
\end{equation*}

Then we have the following famous theorem due to Artin:

\begin{theorem}\cite[Proposition~1.3,Theorem~1.4]{MR1785410}\label{Artin_existence}
Let $K$ be a global function field and $\mathfrak{m}$ a non-zero ideal in $\C O_K$. Then there is an abelian field extension $K_{\mathfrak{m}}|K$ which is unramified outside of prime divisors of $\mathfrak{m}$ such that ${\rm Cl}^{\mathfrak{m}}_K$ is isomorphic to ${\rm Gal}(K_{\mathfrak{m}}|K)$ with an isomorphism given by 
$(\cdot,K_{\mathfrak{m}}|K):{\rm Cl}^{\mathfrak{m}}_K\to{\rm Gal}(K_{\mathfrak{m}}|K),\mathfrak{a}\mapsto(\mathfrak{a},K_{\mathfrak{m}}|K)$.
\end{theorem}

\begin{remark}
Fixing an algebraic closure $\overline{K}$ of $K$, the field extension $K_{\mathfrak{m}}|K$ is also unique, when considered in $\overline{K}.$
\end{remark}

\subsection{Finite fields and constant field extensions}

In the proofs to follow we will need various statements concerning finite fields and so-called constant field extensions of global function fields. In particular, we need the following two lemmas describing the composita and intersections of finite fields and constant field extensions. While these lemmas are standard, we will recall their proofs for the convenience of the reader.

\begin{lemma}\label{finite_field_lemma_1}
Let $\mathbb{E},\mathbb{E}_1,\mathbb{E}_2,\mathbb{E}_3,\mathbb{F}$ be finite fields such that $\mathbb{F}\subset\mathbb{E}_i\subset\mathbb{E}$ holds for $i=1,2,3.$ Further, set $a_i:=[\mathbb{E}_i:\mathbb{F}]$ for $i=1,2.$ Then $\mathbb{E}_1\cdot\mathbb{E}_2$ and $\mathbb{E}_1\cap\mathbb{E}_2$ are the unique subextensions $L,K$ of $\mathbb{E}|\mathbb{F}$ with $[L:\mathbb{F}]={\rm lcm}(a_1,a_2)$ and $[K:\mathbb{F}]={\rm gcd}(a_1,a_2)$ respectively. In particular, if $\mathbb{E}_3\cap\mathbb{E}_i=\mathbb{F}$ holds for $i=1,2$, then $\mathbb{E}_3\cap(\mathbb{E}_1\cdot\mathbb{E}_2)=\mathbb{F}$ holds as well. 
\end{lemma}

\begin{proof}
The first step is to note that for an extension $\mathbb{E}|\mathbb{F}$ of finite fields with $[\mathbb{E}:\mathbb{F}]=a$ that for each divisor $b$ of $a$ there is a unique field $\mathbb{F}_b$ with $\mathbb{F}\subset\mathbb{F}_b\subset\mathbb{E}$ and $[\mathbb{F}_b:\mathbb{F}]=b$: Namely, set $q:=\mathbb{F}.$ Then $\mathbb{E}$ is the splitting field of the polynomial $f_a(Z):=Z^{q^a}-Z$ defined over $\mathbb{F}.$ But $f_b(Z):=Z^{q^b}-Z$ is a divisor of $f_a(Z)$ and hence the splitting field $\mathbb{F}_b$ of $f_b$ must be contained in $\mathbb{E}.$ Further, the sum and products of roots of $f_b(Z)$ are also roots of $f_b(Z).$ Hence $\mathbb{F}_b$ consists precisely of the $q^b$ many distinct roots of $f_b(Z).$ Thus $[\mathbb{F}_b:\mathbb{F}]=b$ must hold. On the other hand, if there is a different field $\mathbb{F}_b'$ with $[\mathbb{F}_b':\mathbb{F}]=b,$ then $x^{q^b}-x=0$ must hold for all $x\in\mathbb{F}_b'$ and so $\mathbb{F}_b'$ must also be contained in $\mathbb{F}_b.$ We next note that $[\mathbb{E}_1\cdot\mathbb{E}_2:\mathbb{F}]$ is a multiple of both $[\mathbb{E}_1:\mathbb{F}]=a_1$ and $[\mathbb{E}_2:\mathbb{F}]=a_2$ and so of ${\rm lcm}(a_1,a_2).$ Thus $\mathbb{E}_1\cdot\mathbb{E}_2$ contains $\mathbb{F}_{{\rm lcm}(a_1,a_2)}.$ But $\mathbb{E}_1=\mathbb{F}_{a_1}$ and $\mathbb{E}_2=\mathbb{F}_{a_2}$ must also be contained in $\mathbb{F}_{{\rm lcm}(a_1,a_2)}$. Hence $\mathbb{E}_1\cdot\mathbb{E}_2$ is also contained in $\mathbb{F}_{{\rm lcm}(a_1,a_2)}.$ Thus $\mathbb{F}_{{\rm lcm}(a_1,a_2)}=\mathbb{E}_1\cdot\mathbb{E}_2$ and so $[\mathbb{E}_1\cdot\mathbb{E}_2:\mathbb{F}]={\rm lcm}(a_1,a_2)$ holds. On the other hand, note that there is a unique field $\mathbb{F}_{{\gcd}(a_1,a_2)}$ contained in $\mathbb{E}$ with $[\mathbb{F}_{{\gcd}(a_1,a_2)}:\mathbb{F}]={\gcd}(a_1,a_2).$ But $\mathbb{F}_{{\gcd}(a_1,a_2)}$ must also be contained in both $\mathbb{E}_1$ and $\mathbb{E}_2$. Thus $\mathbb{F}_{{\gcd}(a_1,a_2)}\subset\mathbb{E}_1\cap\mathbb{E}_2.$ But $[\mathbb{E}_1\cap\mathbb{E}_2:\mathbb{F}]$ must divide both $a_1$ and $a_2$ and hence ${\gcd}(a_1,a_2)$. This implies $\mathbb{F}_{{\gcd}(a_1,a_2)}=\mathbb{E}_1\cap\mathbb{E}_2.$ For the last claim, note that $\mathbb{E}_3\cap\mathbb{E}_i=\mathbb{F}$ for $i=1,2$ implies by the second claim that $a_3$ must be coprime to $a_1$ and $a_2.$ But then 
\begin{equation*}
[\mathbb{E}_3\cap(\mathbb{E}_1\cdot\mathbb{E}_2):\mathbb{F}]={\rm gcd}(a_3,[\mathbb{E}_1\cdot\mathbb{E}_2:\mathbb{F}])={\rm gcd}(a_3,{\rm lcm}(a_2,a_3)) 
\end{equation*}
must be $1$ and so the last claim follows.
\end{proof}

\begin{lemma}\label{finite_field_lemma_2}
Let $K$ be a global function field with $\mathbb{F}$ its field of constants. Further, let $\mathbb{E}_1,\mathbb{E}_2$ be two finite field extension of $\mathbb{F}$ contained in some algebraic closure of $K$. Define $K_1$ as the composita of $K\cdot\mathbb{E}_1$ and define $K_2$ similarly. Then $K_1\cdot K_2=K\cdot(\mathbb{E}_1\cdot\mathbb{E}_2)$ and $K_1\cap K_2=K\cdot(\mathbb{E}_1\cap\mathbb{E}_2)$ holds. 
\end{lemma}	

\begin{proof}
The first statement is clear. For the second statement set $a_i:=[\mathbb{E}_i:\mathbb{F}]$ for $i=1,2$. According to \cite[Proposition~8.1]{MR1876657}, we have $[K_i:K]=a_i$ for $i=1,2$ and consequently the second statement follows from Lemma~\ref{finite_field_lemma_1}.
\end{proof}	

\section{Bounded generation by root elements in positive characteristic}
\label{proof_main}

We introduce the following notations: For a rational prime $p$ and a natural number $n,$ we set $v_p(n)$ as the maximal $i\in\mathbb{N}_0$ such that $p^i$ divides $n.$ Further, for an element $b\in\C O_K-\{0\},$ we define $\epsilon(b)$ as the exponent of the finite group $(\C O_K/b\C O_K)^*.$ The crucial step to derive Theorem~\ref{main_thm} are to prove global function field versions of the two technical lemmas \cite[Lemma~3,4]{MR704220} of Carter and Keller. We begin with \cite[Lemma~3]{MR704220}:  

\begin{lemma}\label{tech_lemma_3}
Let $K$ be a global function field and with $\mathbb{F}$ its field of constants and set $m:=|\mathbb{F}|-1.$ Then up to multiplication by three elementary matrices in ${\rm SL}_2(\C O_K)$ the first row of each $A\in{\rm SL}_2(\C O_K)$ can be assumed to have the form $(a^m,b)$ with $b\C O_K$ a prime ideal.
\end{lemma}

\begin{proof}
Let 
\begin{equation*}
A=
\begin{pmatrix}
a_1 & b_1\\
c_1 & d_1
\end{pmatrix}
\end{equation*}
be given. If $a_1=0,$ then $c_1$ is a unit in $R$ and hence picking any principal, prime ideal $b\C O_K$, the matrix $E_{12}(c_1^{-1})AE_{12}(-b_1+b)$ has the desired form. The case $b_1=0$ is resolved in a similar manner. So assume wlog. that $a_1,b_1\neq 0.$ Let $S=\{R_{\infty}^{(1)},\dots,R_{\infty}^{(L)}\}$ be the set of primes covering the prime $\C P_{\infty}$ of $\mathbb{F}(T)$. Then according to \cite[A16]{MR244257}, for the function
\begin{equation*}
\left(\frac{\cdot,\cdot}{R_{\infty}^{(1)}}\right)_m:K_{R_{\infty}^{(1)}}^*\times K_{R_{\infty}^{(1)}}^*\to\mathbb{F}^*
\end{equation*}
there are $u,w\in K_{R_{\infty}^{(1)}}^*$ such that 
\begin{equation*}
\left(\frac{u,w}{R_{\infty}^{(1)}}\right)_m
\end{equation*}
generates $\mathbb{F}^*.$ Using Dirichlet's Theorem, we can choose a principal prime ideal $a_2\C O_K$ satisfying the following three properties: 
\begin{align*}
a_2\equiv a_1\text{ mod }b_1\C O_K,\forall i=2,\dots,L:a_2\equiv 1\text{ mod }R_{\infty}^{(i)}\text{ and }a_2\equiv w\text{ mod }(R_{\infty}^{(1)})^N
\end{align*}
with $N$ chosen so large that $a_2/w$ has an m.th root in $K_{R_{\infty}^{(1)}}.$ Then note 
\begin{align*}
\left(\frac{u,a_2}{R_{\infty}^{(1)}}\right)_m=\left(\frac{u,a_2/w}{R_{\infty}^{(1)}}\right)_m\cdot\left(\frac{u,w}{R_{\infty}^{(1)}}\right)_m=\left(\frac{u,w}{R_{\infty}^{(1)}}\right)_m.
\end{align*}
Hence there is a positive integer $k$ such that 
\begin{align*}
\left(\frac{b_1,a_2}{R_{\infty}^{(1)}}\right)_m\cdot\left(\frac{u,a_2}{R_{\infty}^{(1)}}\right)_m^k=1.
\end{align*}
Set $v:=u^k.$ Then using Dirichlet's Theorem again, we may choose a principal prime ideal $b_2\C O_K$ satisfying the following three properties
\begin{align*}
b_2\equiv b_1\text{ mod }a_2\C O_K,\forall i=2,\dots,L:b_2\equiv 1\text{ mod }R_{\infty}^{(i)}\text{ and }b_2\equiv v\text{ mod }(R_{\infty}^{(1)})^N
\end{align*}
Again, the integer $N$ is chosen so large that $b_2/v$ is an m.th root of unity in $K_{R_{\infty}^{(1)}}.$ Then according to the m.th power reciprocity law, one has
\begin{equation*}
1=\prod_{\C P} \left(\frac{b_2,a_2}{\C P}\right)_m
\end{equation*}
with the product taken over all primes in $K.$ But note that due to the choice of $a_2\C O_K,b_2\C O_K$ as prime ideals and the fact that $a_2$ and $b_2$ are both chosen to be congruent to $1$ modulo the primes $R_{\infty}^{(i)}$ for $i\geq 2,$ this product reduces to 
\begin{equation*}
1=\left(\frac{b_2,a_2}{b_2\C O_K}\right)_m\cdot\left(\frac{b_2,a_2}{a_2\C O_K}\right)_m\cdot\left(\frac{b_2,a_2}{R_{\infty}^{(1)}}\right)_m
\end{equation*} 
according to \cite[A16]{MR244257}. But as $b_2/v$ is an m.th root in $K_{R_{\infty}^{(1)}},$ this yields
\begin{equation*}
\left(\frac{b_2,a_2}{a_2\C O_K}\right)_m\cdot\left(\frac{b_2,a_2}{R_{\infty}^{(1)}}\right)_m=
\left(\frac{b_2,a_2}{a_2\C O_K}\right)_m\cdot\left(\frac{v,a_2}{R_{\infty}^{(1)}}\right)_m=
\left(\frac{b_2,a_2}{a_2\C O_K}\right)_m\cdot\left(\frac{u^k,a_2}{R_{\infty}^{(1)}}\right)_m=
\left(\frac{b_2,a_2}{a_2\C O_K}\right)_m\cdot\left(\frac{u,a_2}{R_{\infty}^{(1)}}\right)_m^k=1
\end{equation*} 
This implies 
\begin{equation*}
\left(\frac{a_2}{b_2\C O_K}\right)_m^{-1}=\left(\frac{a_2}{b_2\C O_K}\right)_m^{-{\rm ord}_{b_2\C O_K}(b_2\C O_K)}=\left(\frac{b_2,a_2}{b_2\C O_K}\right)_m=1
\end{equation*}
Hence $a_2$ is an m.th power modulo $b_2\C O_K$ and so there is an $a\in\C O_K$ such that $a^m\equiv a_2\text{ mod }b_2\C O_K.$ To finish note that as $a_2-a_1\in b_1\C O_K, b_2-b_1\in a_2\C O_K$ and $a^m-a_2\in b_2\C O_K$, there are $x,y,z\in\C O_K$ such that $a_2=a_1+xb_1,b_2=b_1+ya_2$ and $a^m=a_2+zb_2.$ Hence the matrix $AE_{21}(x)E_{12}(y)E_{21}(z)$ has the desired form.
\end{proof}
\begin{remark}
Note, that the bound in the global function field case is three rather than four as in the number field case \cite[Lemma~3]{MR704220}. The reason for this is that one of the "preprocessing" steps from the number field case served to make $b_1$ prime as integers to the number $m$ of roots of unity in $K$. However, this step has no analog for global function fields. 
\end{remark}
  
Next, we will derive the global function field version of \cite[Lemma~4]{MR704220}:	
	
\begin{lemma}\label{tech_lemma_4}
Let $K$ be a global function field with $\mathbb{F}$ its field of constants. Let $b\in\C O_K-\{0\}$ and a principal ideal $\mathfrak{a}\subset\C O_K$ be given such that $b\C O_K$ is a prime ideal and $\mathfrak{a}+b\C O_K=\C O_K$. Then for every $u\in\C O_K^*,$ there exists an element $c\in\C O_K$ such that $bc\equiv u\text{ mod }\mathfrak{a}$ and such that ${\rm gcd}(\epsilon(b),\epsilon(c))=|\mathbb{F}|-1=m.$
\end{lemma}		
				
\begin{proof}
First, we consider the finite field $\mathbb{E}:=\C O_K/b\C O_K.$ For each rational prime $p\in\mathbb{N}$ choose $e_p\in\mathbb{N}_0$ maximal such that $p^{e_p}$ divides $m.$ Note that $(\C O_K/b\C O_K)^*$ contains $\mathbb{F}^*$ as a subgroup and so $p^{e_p}$ divides $|(\C O_K/b\C O_K)^*|=\epsilon(b).$ Thus $v_p(\epsilon(b))\geq e_p$ holds for all rational primes $p$. Set $S:=\{p_1,\dots,p_l\}$ as the finite set of primes $p$ such that $v_p(\epsilon(b))>e_p.$ If $S=\emptyset$, then $\epsilon(b)=m$ and so picking any element $c\in\C O_K$ with $bc\equiv u\text{ mod }\mathfrak{a}$ we are done. So we assume in the following that $S\neq\emptyset.$ For each $p\in S$, pick a primitive $p^{e_p+1}$.th root of unity $\xi_p$ in $\mathbb{E}$ and consider the finite fields $\mathbb{E}_i:=\mathbb{F}[\xi_{p_i}]$ for $i=1,\dots,l$. We will next inductively describe a process to find subfields $\mathbb{L}_1^{(t)},\dots,\mathbb{L}_t^{(t)}$ of $\mathbb{E}_1,\dots,\mathbb{E}_t$ for all $t\leq l$ such that $\mathbb{L}_i^{(t)}\cap \mathbb{L}_j^{(t)}=\mathbb{F}$ or $\mathbb{L}_i^{(t)}=\mathbb{L}_j^{(t)}\neq\mathbb{F}$ holds for all $1\leq i,j\leq t.$ First, if $t=1$, there is nothing to do. So assume $t>1$ and assume further that the subfields $\mathbb{L}_1^{(t-1)},\dots,\mathbb{L}_{t-1}^{(t-1)}$ are already chosen as required. Then we check whether there is an $i$ between $1$ and $t-1$ such that $\mathbb{L}_i^{(t-1)}\cap\mathbb{E}_t\neq\mathbb{F}.$ If there is such an $i$, then we set $\mathbb{L}_i^{(t)}:=\mathbb{L}_t^{(t)}:=\mathbb{L}_i^{(t-1)}\cap\mathbb{E}_t$ for all such $i\leq t$ and set $\mathbb{L}_h^{(t)}:=\mathbb{L}_h^{(t-1)}$ for all other $h\leq t-1$. However, if $\mathbb{E}_t\cap\mathbb{L}_i^{(t-1)}=\mathbb{F}$ holds for all $i\leq t$, then we set $\mathbb{L}_i^{(t)}:=\mathbb{L}_i^{(t-1)}$ for $i\leq t-1$ and $\mathbb{L}_t^{(t)}:=\mathbb{E}_t.$ Clearly, these choices of $\mathbb{L}_i^{(t)}$ have the required properties. Next, consider the constant field extensions $K_1:=\mathbb{L}_1^{(l)}\cdot K,\dots,K_l:=\mathbb{L}_l^{(l)}\cdot K$ of $K$. For simplicity, we will also assume in the following that all the fields $K_1,\dots,K_l$ are actually distinct. Also consider the compositum $L:=\prod_{i=1}^l K_l.$ Next, consider the field $B_l:=K_l\cap\prod_{i=1}^{l-1} K_i.$ Note that according to Lemma~\ref{finite_field_lemma_2}, we know that
\begin{align*}
B_l=K_l\cap\prod_{i=1}^{l-1}K_i=(K\cdot\mathbb{L}_l^{(l)})\cap(\prod_{i=1}^{l-1}K\cdot\mathbb{L}_i^{(l)})=K\cdot(\mathbb{L}_l^{(l)}\cap\prod_{i=1}^{l-1}\mathbb{L}_i^{(l)}).
\end{align*}
However, by the construction of the $\mathbb{L}_i^{(l)}$, we know that $\mathbb{L}_l^{(l)}\cap\mathbb{L}_i^{(l)}=\mathbb{F}$ holds for all $i=1,\dots,l-1$ and hence  
$(\mathbb{L}_l^{(l)}\cap\prod_{i=1}^{l-1}\mathbb{L}_i^{(l)})=\mathbb{F}$ according to the last statement in Lemma~\ref{finite_field_lemma_1}. Thus $B_l=K.$ Hence we have an isomorphism of Galois groups:
\begin{equation*}
{\rm Gal}(L\mid K)\to{\rm Gal}(K_l|K)\times{\rm Gal}(\prod_{i=1}^{l-1}K_i|K),\sigma\mapsto(\sigma|_{K_l},\sigma|_{\prod_{i=1}^{l-1}K_i}).
\end{equation*}
Iterating this argument, we obtain that ${\rm Gal}(L|K)=\prod_{i=1}^l{\rm Gal}(K_i|K).$ Also note that $L$ is a constant field extension of $K$ and hence an unramified field extension according to \cite[Proposition~8.5]{MR1876657}. Next, consider the ray class field $K_{\mathfrak{a}}$ associated to the ideal $\mathfrak{a}.$ By assumption $b\C O_K$ and $\mathfrak{a}$ are coprime and so $b\C O_K$ is unramified in $K_{\mathfrak{a}}|K.$ Hence $b\C O_K$ is also unramified in the extension $L_{\mathfrak{a}}|K$ for $L_{\mathfrak{a}}:=K_{\mathfrak{a}}\cdot L.$ Thus we may consider the element $\sigma:=(b\C O_K,L_{\mathfrak{a}}|K)\in{\rm Gal}(L_{\mathfrak{a}}|K).$ Next, we define the constant field extensions $L_p:=K\cdot\mathbb{F}[\xi_p]$ of $K$ for $p\in S$ and consider the restrictions $\sigma_p:=\sigma|_{L_p}=(b\C O_K,L_p|K)$ for $p\in S.$ 
\begin{claim}
One has $\sigma_p={\rm id}_{L_p}$ for $p\in S.$
\end{claim}
Let $Q$ be a prime in $\C O_{L_p}$ covering $b\C O_K.$ Then $\C O_{L_p}/Q=(\C O_K/b\C O_K)[\overline{\xi}_p]$ holds for $\overline{\xi}_p$ the image of $\xi_p$ in $\C O_{L_p}/Q$ according to \cite[Proposition~8.10]{MR1876657}. But note that $v_p(\epsilon(b))>e_p$ and so $p^{e_p+1}$ divides $\epsilon(b)=|(\C O_K/b\C O_K)^*|.$ This implies in particular that $\overline{\xi}_p^{|(\C O_K/b\C O_K)^*|}=1$ holds and thus $\overline{\xi}_p$ is a root of the polynomial $F(z):=z^{|(\C O_K/b\C O_K)^*|}-1\in (\C O_K/b\C O_K)[z].$ But the roots of this polynomial are precisely the elements of $(\C O_K/b\C O_K)^*$ and so $\overline{\xi}_p$ is an element of $(\C O_K/b\C O_K)^*.$ Thus $\C O_{L_p}/Q=\C O_K/b\C O_K$. Hence the relative degree $f(Q|b\C O_K)$ is equal to $1.$ But then \cite[Proposition~9.10]{MR1876657} implies $\sigma_p=(b\C O_K,L_p|K)=1.$ 

Next, note that for $i=1,\dots,l,$ each group ${\rm Gal}(K_i|K)={\rm Gal}(\mathbb{L}_i^{(l)}|\mathbb{F})$ is non-trivial by construction and so as ${\rm Gal}(L|K)=\prod_{i=1}^l{\rm Gal}(K_i|K)$ is a quotient of ${\rm Gal}(L_{\mathfrak{a}}|K)$, we may choose an element $\sigma_1\in{\rm Gal}(L_{\mathfrak{a}}|K)$ with $\sigma_1|_{K_i}\neq{\rm id}_{K_i}$ for all $i=1,\dots,l.$ Set further $\sigma_2:=\sigma\sigma_1^{-1}.$ Then according to Tchebotarevs density theorem \cite[Theorem~9.13A]{MR1876657}, we may choose two primes $\C P_1,\C P_2$ unramified in $L_{\mathfrak{a}}|K$ such that $(\C P_j,L_{\mathfrak{a}}|K)=\sigma_j^{-1}$ holds for $j=1,2.$ Note that as the $L_p$ for $p\in S$ are constant field extensions of $K,$ the primes $\C P_j$ are automatically unramified in $L_p|K$ for all $p\in S.$ So we can also consider $(\C P_j,L_p|K)\in{\rm Gal}(L_p|K)$ for all $p\in S.$ But note that by the construction of the $K_i$, there is an $i=1,\dots,l$ such that $K_i\subset L_p$. But $(\C P_j,L_p|K)$ maps to $(\C P_j,K_i|K)=\sigma_j|_{K_i}\neq{\rm id}_{K_i}$ under the restriction homomorphism ${\rm Gal}(L_p|K)\to{\rm Gal}(K_i|K).$ Hence $(\C P_j,L_p|K)=\sigma_j|_{L_p}^{-1}\neq{\rm id}_{L_p}$ holds for $j=1,2$ and $p\in S.$ Thus \cite[Proposition~9.10]{MR1876657} implies that neither $\C P_1$ nor $\C P_2$ split completely in the extension $L_p|K$ for any $p\in S.$ But similar to the proof of the claim above this implies that $v_p(|\C O_K/\C P_j|-1)=e_p$ for $j=1,2$ and $p\in S$, because $v_p(|\C O_K/\C P_j|-1)>e_p$ would imply that $\C P_j$ does split completely in the unramified field extension $L_p|K$. Next, choose $c_0\in \C O_K$ such that $bc_0\equiv u\text{ mod }\mathfrak{a}.$ But this implies 
\begin{align*}
(\C P_1\cdot \C P_2,K_{\mathfrak{a}}|K)&=(\C P_1,K_{\mathfrak{a}}|K)\cdot(\C P_2,K_{\mathfrak{a}}|K)=(\sigma_1^{-1}|_{K_{\mathfrak{a}}})\cdot(\sigma_2^{-1}|_{K_{\mathfrak{a}}})=\sigma|_{K_{\mathfrak{a}}}^{-1}=(b\C O_K,K_{\mathfrak{a}}|K)^{-1}\\
&=(c_0\C O_K,K_{\mathfrak{a}}|K).
\end{align*}
Next, Theorem~\ref{Artin_existence} implies that there is a $\lambda\in K^*$ multiplicatively congruent to $1$ modulo $\mathfrak{a}$ such that $c_0\C O_K=\lambda\C P_1\cdot\C P_2.$ So setting $c:=c_0\lambda^{-1}$, one obtains $c\C O_K=\C P_1\cdot\C P_2.$ Hence $\epsilon(c)={\rm lcm}(\epsilon(\C P_1),\epsilon(\C P_2))$ holds. Then let $p$ be a rational prime. If $p\in S,$ then $v_p(\epsilon(c))=\min\{v_p(\epsilon(\C P_1)),v_p(\epsilon(\C P_2))\}=e_p.$ On the other hand, if $p\notin S,$ then $v_p(\epsilon(b))=e_p$. Hence $v_p({\rm gcd}(\epsilon(b),\epsilon(c)))=e_p$ holds for all rational primes $p$ and so ${\rm gcd}(\epsilon(b),\epsilon(c))=m.$
\end{proof}	 
	
\begin{remark}
We would like to consider the field extension $\mathbb{F}[\xi_p\mid p\in S]\cdot K$ of $K$ in the above argument and show that ${\rm Gal}(\mathbb{F}[\xi_p\mid p\in S]\cdot K|K)$ decomposes as the direct product $\prod_{p\in S}{\rm Gal}(L_p|K).$ However, this strategy doesn't work as the various $L_p$ do in some cases intersect in non-trivial ways. The construction of the $K_i$ is done to circumvent precisely this problem.
\end{remark}	
	
Using these two lemmas, one can now prove Theorem~\ref{main_thm}:

\begin{proof}
Let $A\in{\rm SL}_n(\C O_K).$ Then using the stability argument from the proof of \cite[Main Theorem]{MR704220}, one can transform $A$ into an element of the form 
\begin{equation*}
A_1=
\begin{pmatrix}
\begin{matrix}
a_1 & b_1\\
c_1 & d_1
\end{matrix}
& 0\\
0 & I_{n-2}
\end{pmatrix}
\end{equation*} 
by multiplication with $1/2(3n^2-n-10)$ elementary matrices. Using Lemma~\ref{tech_lemma_3}, one can then transform $A_1$ into a matrix 
\begin{equation*}
A_2=
\begin{pmatrix}
\begin{matrix}
a^m & b\\
c_2 & d_2
\end{matrix}
& 0\\
0 & I_{n-2}
\end{pmatrix}
\end{equation*} 
by multiplication with three elementary matrices. But note that by Lemma~\ref{tech_lemma_4}, we can find an element $c\in\C O_K$ such that $bc\equiv -1\text{ mod }a\C O_K$ and such that ${\rm gcd}(\epsilon(b),\epsilon(c))=m.$ Consequently, there is an element 
\begin{equation*}
A_3=
\begin{pmatrix}
\begin{matrix}
a & b\\
c & d_3
\end{matrix}
& 0\\
0 & I_{n-2}
\end{pmatrix}^m
\end{equation*} 
of ${\rm SL}_n(\C O_K).$ Then according to \cite[Lemma~1]{MR704220}, one can transform $A_2$ into $A_3$ by multiplication with $16$ elementary matrices and according to \cite[Lemma~2]{MR704220}, the matrix $A_3$ can be transformed into the identity matrix by multiplication with $36$ elementary matrices. Thus in total, the initial matrix $A$ is a product of at most $1/2(3n^2-n-10)+3+16+36=1/2(3n^2-n)+50$ elementary matrices.  
\end{proof}	

	
\section{What remains}

After some research activity in the last fifteen or so years bounded generation by root elements for split Chevalley groups defined over rings of integers in global fields is now essentially understood: On the qualitative side, the Carter-Keller-Paige approach presented by Morris \cite{MR2357719} serves as a uniform framework to prove bounded generation results first for ${\rm SL}_n$ of number fields and as seen in \cite{Chevalley_positive_char_tentative} also for global function fields and the other split Chevalley groups.   

On the quantitative side, the main question previously was whether there are only bounds for elementary bounded generation depending on the global field in question or not. We denote the minimal number of elementary matrices needed to write any element of ${\rm SL}_n(R)$ by $\nu_n(R)$. For $n\geq 3,$ standard stable range considerations yield 
\begin{equation*}
\nu_n(R)\leq 1/2(3n^2-n)-27+\nu_3(R)
\end{equation*}
for $R$ the ring of S-algebraic integers in a global field $K$ with $S$ a non-empty set of valuations of $K$ containing all archimedean places of $K.$ We fix $K,S$ and $R$ like this in the following. As seen by Rapinchuk, Morgan and Sury \cite{MR3892969} if $K$ is a number field but not imaginary quadratic or $\mathbb{Z}$, one has $\nu_3(R)\leq 16$. This preprint establishes in Theorem~\ref{main_thm} for a global function field $K$ that $\nu_3(R)\leq 62$ holds. The remaining cases are the imaginary quadratic integers and $\mathbb{Z}$. But the ring of imaginary quadratic integers $R_D$ in the number field $\mathbb{Q}[\sqrt{-D}]$ are the rings of all algebraic integers of a number field of degree two over $\mathbb{Q}.$ Thus \cite[Theorem~1.2]{MR2357719} implies that there is a natural number $L'(2,3)$ independent of the precise $D$ such that $\nu_3(R_D)\leq L'(2,3)$. Thus setting $L(2,3):=\max\{L'(2,3),62,v_3(\mathbb{Z})\}-27$, we obtain the following:

\begin{theorem}\label{really?}
There is an $L(2,3)\in\mathbb{N}$ such that for each global field $K$, each non-empty set $S$ of valuations of $K$ containing all its archimedean places and $R$ the ring of all $S$-algebraic integers in $K$, the inequality $\nu_n(R)\leq L(2,3)+1/2(3n^2-n)$ holds for all $n\geq 3$.  
\end{theorem} 
 
So bounded elementary generation has bounds completely independent of the precise global field $K$ and the precise set $S$. This answers questions raised in \cite{MR3892969} and mentioned in a preprint by Kunyavskii, Plotkin and Vavilov \cite[9.~Final Remarks]{https://doi.org/10.48550/arxiv.2204.10951}. On the other hand, the precise values for $\nu_n(R)$ themselves remain a mystery: A simple counting argument shows that $\nu_n(R)$ being quadratic in $n$ is at least the correct asymptotic, but the bound $L(2,3)$ above is non-explicit as it arises (among else) from a model-theoretic compactness argument and it would be desirable to replace it with some explicit number independent of $D.$ Additionally, it is very unpleasant that three very different line of arguments are needed to seperately treat the three cases of global function fields, imaginary quadratic number fields and the rest. Thus it stands to reason that there is a better method to treat them all at once. Given Theorem~\ref{really?}, one might even nurture the following hope:

\begin{conjecture}
There is a function $f:\mathbb{N}\to\mathbb{N}$ such that for any $K,S$ and $R$ as in Theorem~\ref{really?}, one has $\nu_n(R)=f(n)$ for all $n\geq 3.$
\end{conjecture}

Lastly, we should mention that while the main result in this manuscript, Theorem~\ref{main_thm}, is only stated for ${\rm SL}_n$, it is all but clear that the result will also generalize to all the other split Chevally groups.


\printbibliography


\end{document}